\documentclass{article}
\usepackage{amsmath,amssymb,amsthm,amsfonts}
\usepackage{hyperref}
\usepackage{algorithm,algorithmic}
\usepackage{tikz-cd}
\usepackage{verbatim}

\newcommand{\Trace}{\operatorname{Trace}}
\newcommand{\Norm}{\operatorname{Norm}}
\newcommand{\Diff}{\operatorname{Diff}}
\newcommand{\Disc}{\operatorname{Disc}}
\newcommand{\Hom}{\operatorname{Hom}}
\newcommand{\End}{\operatorname{End}}
\newcommand{\Gal}{\operatorname{Gal}}
\newcommand{\Res}{\operatorname{Res}}
\newcommand{\Pic}{\operatorname{Pic}}
\newcommand{\Jac}{\operatorname{Jac}}
\newcommand{\goodenough}{ideal}

\newcommand{\CC}{\mathbb{C}}

\newcommand{\QQ}{\mathbb{Q}}
\newcommand{\ZZ}{\mathbb{Z}}
\newcommand{\FF}{\mathbb{F}}

\newcommand{\sO}{\mathcal{O}}

\theoremstyle{plain}
\newtheorem{thm}{Theorem}
\newtheorem{theorem}[thm]{Theorem}
\newtheorem{proposition}[thm]{Proposition}
\newtheorem{lemma}[thm]{Lemma}
\newtheorem{corollary}[thm]{Corollary}

\theoremstyle{remark}
\newtheorem{remark}[thm]{Remark}

\theoremstyle{definition}
\newtheorem{definition}[thm]{Definition}
\newtheorem{example}[thm]{Example}

\title{Products and Polarizations of Super-Isolated Abelian Varieties}

\author{Stefano Marseglia, Travis Scholl}

\AtEndDocument{\bigskip{\footnotesize%
  S.~Marseglia, \textsc{Mathematical Institute, Utrecht University, Utrecht, The Netherlands} \par
  \textit{E-mail}: \texttt{{s.marseglia@uu.nl}} \par\ \par
  T.~Scholl, \textsc{University of California, Irvine} \par
  \textit{E-mail}: \texttt{{traviswscholl@gmail.com}}
}}

\begin{document}

\maketitle

\begin{abstract}
In this paper we study super-isolated abelian varieties, that is, abelian varieties over finite fields whose isogeny class contains a single isomorphism class.
  The goal of this paper is to (1) characterize whether a product of super-isolated varieties is super-isolated, and (2) characterize which super-isolated abelian varieties admit principal polarizations, and how many up to polarized isomorphisms.
\end{abstract}

\section{Introduction}

An abelian variety $A/\FF_q$ is \emph{super-isolated} if its $\FF_q$-isogeny class contains only the isomorphism class of $A$.
Hence a super-isolated $A/\FF_q$ is determined up to isomorphism by the characteristic polynomial of its Frobenius endomorphism.
In \cite{scholl2018super}, the second author introduced super-isolated elliptic curves and surfaces in the context of cryptography, and in \cite{scholl2019super} this was generalized to higher dimensional simple super-isolated abelian varieties.
In this paper we continue on this path.  In particular we focus on two questions:
\begin{enumerate}
  \item \label{q1} When is the product of super-isolated varieties super-isolated?
  \item \label{q2} When does a super-isolated variety admit a principal polarization and, if that is the case, how many are there up to polarized isomorphism?
\end{enumerate}
In \cite{scholl2018super} it is shown that simple super-isolated abelian varieties are rare,
in the sense that the for most finite fields $\FF_q$ there are no super-isolated varieties over $\FF_q$.
Trying to generalize this statement to non-simple abelian varieties leads to the first question.
Our methods show that super-isolated products are even rarer and can be enumerated as efficiently as super-isolated simple varieties.

The second question is related to the problem of determining whether an abelian variety is the Jacobian of a curve.
This is in general a difficult question, see \cite{maisner2002abelian} for the case of abelian surfaces.
In the case of super-isolated abelian varieties all the arithmetic information is encoded in the Weil polynomial determining the isogeny class.
Since every Jacobian admits a canonical principal polarization, the second question is a first step towards characterizing super-isolated Jacobians.

In this paper, we first present some general results about products of super-isolated abelian varieties, see Section \ref{sec:glueing_exp}.
Then, in order to give answers to our questions, we focus on a class of abelian varieties that we call \emph{\goodenough{}}, see Definition \ref{def:good-enough}.
For the category of such abelian varieties, we have functorial descriptions in terms of finitely generated free $\ZZ$-modules with a ``Frobenius-like" endomorphism, see \cite{deligne1969varieties,centeleghe2015categories}.
In the ordinary case, we can describe also dual varieties and polarizations, see \cite{howe1995principally}.
We exploit these descriptions in Theorem \ref{thm:product-siav} where we exhibit a criterion to answer Question \ref{q1}.
It turns out that the technology developed to prove Theorem \ref{thm:product-siav} in Sections~\ref{sec:weil-gen-products} and \ref{sec:siav-products} allows us to show that, for any fixed dimension, there are only finitely many super-isolated \goodenough{} abelian varieties that are not simple, see Corollary~\ref{cor:finite-products-siav}.
In Section~\ref{sec:enum_siav} we give an algorithm to enumerate them and we produce complete lists of such abelian varieties that are a product of elliptic curves and surfaces, see Algorithm~\ref{alg:find-wg-2} and Table~\ref{tab:siav-data}.

In Section~\ref{sec:siav-pol}, we characterize which simple ordinary super-isolated varieties admit a principal polarization, see Theorem~\ref{thm:siav-pp-criteria}.
We also show that if such a polarization exists, then it is unique, see Theorem~\ref{thm:unique-pol}.
Moreover in Corollary \ref{cor:pol_square_free} and Remark \ref{remark:pol_power} we discuss the product case.
These results give an answer to Question \ref{q2}.

In Section \ref{sec:jacobians} we apply the theory developed in the previous sections to prove some properties of super-isolated Jacobians, see Proposition \ref{prop:jac}.

In this paper, all morphisms between abelian varieties over a field $k$ are defined over the same field $k$.

\subsection*{Acknowledgements}
The first author was partially supported by the Max Planck Society (Postdoctoral Fellowship) and by the Dutch Research Council (NWO grant 613.001.651).
The second author was partially supported by Alfred P. Sloan Foundation (grant number G-2014-13575).
The second author would also like to thank Alice Silverberg and Karl Rubin for helpful discussions on abelian varieties and class field theory.
The authors thank Jonas Bergstr\"om, Valentijn Karemaker, Christophe Ritzenthaler and Shahed Sharif for comments, and Everett Howe helpful discussion that lead to the results contained in Section \ref{sec:glueing_exp}.
The authors are grateful to the anonymous reviewer for helpful suggestions and comments.

\section{Products of Weil generators}\label{sec:weil-gen-products}

Weil generators in CM fields have been studied in \cite{scholl2018super} and \cite{scholl2019super}.
They represent the Frobenius endomorphism of a super-isolated abelian variety.
The purpose of this section is to generalize the notion of a Weil generator to a product of CM fields and give quantitative results.
For a CM field $K$ we will denote its CM involution by $\bar{\cdot}$.

\begin{definition}
  Let $K$ be a product of CM fields $K = K_1 \times \cdots \times K_n$.
  We say $\alpha \in K$ is a \emph{Weil generator for $K$} if $\alpha\overline{\alpha} \in \ZZ$ and $\sO_K = \ZZ[\alpha,\overline{\alpha}]$, where $\sO_K = \prod \sO_{K_i}$.
  That is, $\alpha\overline{\alpha}$ lies in the image of the diagonal embedding $\ZZ \to K$, and the subring generated by $\alpha$ and $\overline{\alpha}$ is the integral closure of $\ZZ$ in $K$.
\end{definition}

Let $K$ be a CM-field. Denote by $F$ the fixed field of the CM involution of $K$.
Fix $\gamma \in K$ satisfying $\sO_K = \sO_{F}[\gamma]$.
Let $R$ be the set of all $\eta \in F$ such that $\sO_F = \ZZ[\eta]$.
Choose a set $T \subset R$ of representatives of $R/\sim$, where $\eta_1 \sim \eta_2$ if $\eta_1 - \eta_2 \in \ZZ$.
If $\gamma$ does not exist or $T = \emptyset$, then $K$ has no Weil generators.
Indeed, if $\alpha$ is a Weil generator for $K$ then we could choose $\gamma = \alpha$ and $(\alpha+\overline{\alpha}) \in R$, see \cite[Lemma~3.13]{scholl2019super}.
Given the order $\sO_F$, the set $T$ is always finite and can be effectively computed, see \cite{gyory1976}.
By \cite[Lemma~3.15]{scholl2019super} every Weil generator $\alpha \in K$ can be written as
\begin{equation}\label{eq:alpha_u_gamma_eta_a}
\alpha = \frac{u(\gamma - \overline{\gamma}) + \eta + a}{2}
\end{equation}
for a unique triple $(u,\eta,a) \in \sO_F^\times \times T \times \ZZ$.

\begin{example}
  Let $K = \QQ(i)$ and put $\gamma = i$ and $T = \{0\}$.
  Then every Weil generator is of the form $(\pm 2i + a)/2$ for some integer $a \in \ZZ$.
  Alternatively, every Weil generator can also be written as $b \pm i$ for some $b \in \ZZ$.
\end{example}

\begin{example}
  Let $K = \QQ(\zeta_5)$ and put $\gamma = \zeta_5$ and $T = \{(1+\sqrt{5})/2,(1-\sqrt{5})/2\}$.
  Then every generator can be expressed as $(u(\zeta_5 - \bar\zeta_5) + (1 \pm \sqrt{5})/2 + a)/2$ for some $u \in \sO_{\QQ(\sqrt{5})}^\times$ and $a \in \ZZ$.
\end{example}

\begin{lemma}\label{lem:weil-gen-to-curve}
  Let $K$ be a CM field with maximal real subfield $F$, and $\alpha \in K$ a Weil generator of $K$.
  Fix $\eta \in T$ and $a\in \ZZ$ as in Equation~\eqref{eq:alpha_u_gamma_eta_a}.
  If $q = \alpha\overline{\alpha}$, then
  \[
    \Norm_{F/\QQ} \left( \left(\eta + a\right)^2 - 4q \right) = \frac{\Disc_{K/\QQ}}{\Disc_{F/\QQ}^2}.
  \]
\end{lemma}
\begin{proof}
  Let $g = \deg F$ and $\beta = \alpha+\overline{\alpha}$.
  By \cite[Ch.2~Ex.~23]{marcus2018number},
  \begin{align*}
    \Disc_{K/\QQ}(1,\dots,\beta^{g-1},\alpha,\dots,\alpha\beta^{g-1})
    =
    \Disc_{F/\QQ}(\beta)^2 \Norm_{F/\QQ}\Disc_{K/F}(\alpha).
  \end{align*}
  But $\{1,\alpha\}$ is a $\sO_{F}$-basis for $\sO_K$ and $\{1,\beta,\dots,\beta^{g-1}\}$ is a $\ZZ$-basis for $\sO_{F}$,
  so this reduces to $\Disc_{K} = \Disc_{F}^2 \Norm_{F/\QQ}\Disc_{K/F}(\alpha)$.
  The result follows because
  \[
    \Disc_{K/F}(\alpha)
    = \det
    \begin{pmatrix}
    \Trace_{K/F}(1) & \Trace_{K/F}(\alpha) \\
    \Trace_{K/F}(\alpha) & \Trace_{K/F}(\alpha^2)
    \end{pmatrix}
    = (\alpha - \overline{\alpha})^2
    = (\eta + a)^2 - 4q.
  \]
\end{proof}

The expression in Lemma~\ref{lem:weil-gen-to-curve} can be viewed as the equation of a plane curve by replacing $a$ and $4q$ with formal variables.
Lemma~\ref{lem:curve-is-irreducible} below together with the substitutions
\[
  \Norm_{F/\QQ} \left( \left(\eta + a\right)^2 - 4q \right)
  =
  \prod_{\sigma:F \to \CC} \left( \left(a + \sigma(\eta)\right)^2 - 4q \right)
\]
and $t = -\Disc_{K}/\Disc_{F}^2$
shows that this curve is geometrically irreducible.

\begin{lemma}\label{lem:curve-is-irreducible}
  Let $t,a_1,\dots,a_n \in \CC$ and
  \[
    P(x,y) = \prod_{i=1}^n \left( \left(x - a_i\right)^2 - y \right) + t.
  \]
  If the $a_i$ are distinct and $t \neq 0$, then $P(x,y)$ is irreducible.
\end{lemma}
\begin{proof}
  Suppose that $P(x,y) = h(x,y)g(x,y)$ and that $h(x,y)$ is non-constant.
  Let $L_i(x,y) = (x - a_i)^2 - y$ so that $P(x,y) = \prod_{i=1}^n L_i(x,y) + t$.
  For any point $Q$ with $L_i(Q) = 0$, we have $P(Q) = t$.
  Hence $h(x,(x-a_i)^2)$ must be constant, say $c_i$.
  Because $L_i$ and $L_j$ both vanish at $((a_i+a_j)/2,((a_i-a_j)/2)^2)$, it follows that $c_1 = c_2 = \cdots = c_n$.
  Let $c$ denote this value.
  Since $h(x,y)-c$ is zero whenever $L_i(x,y)$ is, it follows that $L_i(x,y)$ divides $h(x,y)-c$ for every $i=1,\dots,n$.
  As the $L_i(x,y)$ are distinct and irreducible, we must have that their product also divides $h(x,y) - c$.
  But $h(x,y)$ is non-constant, so $h(x,y)-c \neq 0$ and therefore $\deg h(x,y) = \deg P(x,y)$.
  It follows that $g(x,y)$ is constant.
\end{proof}

\begin{proposition}\label{prop:reduce-to-monogenic}
  Let $K = K_1 \times \cdots \times K_n$ be a product of CM fields.
  For $i=1,\ldots,n$ pick an algebraic integer $\alpha_i \in \sO_{K_i}$ and put $\alpha = (\alpha_1,\dots,\alpha_n)$.
  For each $i$ denote by $g_i(x) \in \ZZ[x]$ the minimal polynomial of $\alpha_i + \overline{\alpha}_i$.
  Then $\ZZ[\alpha,\overline{\alpha}] = \prod \ZZ[\alpha_i,\overline{\alpha}_i]$ if and only if $|\Res(g_i,g_j)| = 1$ for all $i \neq j$.
\end{proposition}
\begin{proof}
  First we will show that $\ZZ[\alpha,\overline{\alpha}] = \prod \ZZ[\alpha_i,\overline{\alpha}_i]$ if and only if $\ZZ[\alpha + \overline{\alpha}] = \prod \ZZ[\alpha_i + \overline{\alpha}_i]$.
  Let $e_i$ be the $i$-th orthogonal idempotent of $K=\prod K_i$.
  Observe that $\ZZ[\alpha + \overline{\alpha}]$ is the subset of $\ZZ[\alpha,\overline{\alpha}]$ that is fixed by complex conjugation.
  Moreover, for every $i=1,\ldots,n$, we have $e_i=\bar e_i$.
  The claim follows from the following chain of equivalences: we have $\ZZ[\alpha,\overline{\alpha}] = \prod \ZZ[\alpha_i,\overline{\alpha}_i]$ if and only if
   $e_i \in \ZZ[\alpha,\overline{\alpha}]$ for all $i=1,\ldots,n$ if and only if
   $e_i \in \ZZ[\alpha + \overline{\alpha}]$ for all $i=1,\ldots,n$ if and only if $\ZZ[\alpha + \overline{\alpha}] = \prod \ZZ[\alpha_i + \overline{\alpha}_i]$.

  Next we claim that $\ZZ[\alpha + \overline{\alpha}] = \prod \ZZ[\alpha_i + \overline{\alpha}_i]$ if and only if $|\Res(g_i,g_j)| = 1$ for all $i \neq j$.
  The map $\ZZ[x] \to \ZZ[\alpha + \overline{\alpha}]$ sending $x$ to $\alpha + \overline{\alpha}$ is surjective.
  Therefore, it is equivalent to show that $\ZZ[x]/\mathrm{lcm}(g_i) \cong \prod \ZZ[x]/g_i$ if and only if $|\Res(g_i,g_j)| = 1$.
  The former holds if and only if the ideals $g_i\ZZ[x]$ are coprime.
  By \cite[Lem.~11.3]{serres1976vaserstein} (see also \cite[p.~420]{resultants1983myerson} for a slightly stronger version), this holds if and only if the pairwise resultants of the $g_i$ are units.
\end{proof}

\begin{corollary}\label{cor:products-of-wg}
  Let $K = K_1 \times \cdots \times K_n$ be a product of CM fields.
  Let $\alpha = (\alpha_1,\dots,\alpha_n) \in K$, and $g_i$ be the minimal polynomial for $\alpha_i + \overline{\alpha}_i$.
  Then $\alpha$ is a Weil generator for $K$ if and only if each $\alpha_i$ is a Weil generator for $K_i$, $\alpha_1\overline{\alpha}_1 = \cdots = \alpha_n\overline{\alpha}_n$, and $|\Res(g_i,g_j)| = 1$ for all $i \neq j$.
\end{corollary}
\begin{proof}
  Suppose that $\alpha$ is a Weil generator for $K$
  Then $\alpha_1\overline{\alpha}_1 = \cdots = \alpha_n\overline{\alpha}_n \in~\ZZ$.
  Moreoever, $\sO_K = \ZZ[\alpha,\overline{\alpha}]$.
  Because $\sO_K = \prod \sO_{K_i}$, it follows that $\sO_{K_i} = \ZZ[\alpha_i,\overline{\alpha}_i]$ for all $i$.
  So the resultant condition in the statement follows from Proposition~\ref{prop:reduce-to-monogenic}.

  Suppose that the latter conditions hold.
  By Proposition~\ref{prop:reduce-to-monogenic}, we have $\ZZ[\alpha,\overline{\alpha}] = \prod \ZZ[\alpha_i,\overline{\alpha}_i] = \prod \sO_{K_i} = \sO_K$.
  Since $\alpha_i\overline{\alpha}_i$ is the same for each $i$, we also have that $\alpha\overline{\alpha}$ lies in the diagonal $\ZZ \to K$.
  Therefore $\alpha$ is a Weil generator for $K$.
\end{proof}

\begin{example}\label{ex:product-of-weil-generators-is-not-enough}
  In Corollary~\ref{cor:products-of-wg}, both the condition that $|\Res(g_i,g_j)| = 1$ for all $i \neq j$, and the condition that $\alpha_1\overline{\alpha}_1=\cdots =\alpha_n\overline{\alpha}_n$ are necessary.
  To see this, let $K_1 = \QQ(\sqrt{-11})$ and $K_2 = \QQ(\sqrt{-19})$.
  Then $\alpha_1 = (3 - \sqrt{-11})/2$ and $\alpha_2 = (1 - \sqrt{-19})/2$ are Weil generators with norm $5$ for $K_1$ and $K_2$ respectively.
  However, $|\Res(g_1,g_2)| = |\Res(x - 3, x - 1)| = 2$.
  Therefore the ring $\ZZ[(\alpha_1,\alpha_2),(\overline{\alpha}_1,\overline{\alpha}_2)] \neq \sO_{K_1 \times K_2}$, so $(\alpha_1,\alpha_2)$ is not a Weil generator for $K_1 \times K_2$.

  If we instead set $K_2 = \QQ(\sqrt{-13})$ and $\alpha_2 = 2 - i\sqrt{13}$ then $|\Res(g_1,g_2)| = |\Res(x - 3, x - 4)| = 1$ but $\Norm(\alpha_2) = 17$.
  So in this case, $(\alpha_1,\alpha_2)$ is not a Weil generator because its norm does not lie in the diagonal embedding $\ZZ \to K_1 \times K_2$.
\end{example}

\begin{lemma}\label{lem:wg-products}
  Let $K_1$ and $K_2$ be CM fields. Then there are finitely many Weil generators in $K_1 \times K_2$.
\end{lemma}
\begin{proof}
  For $i=1,2$, choose $\gamma_i$ and $T_i$ for $K_i$ as in the beginning of Section~\ref{sec:weil-gen-products}.
  If for some $i$, $\gamma_i$ does not exist or $T_i = \emptyset$, then $K_i$ does not have any Weil generators and we are done.
  So we will assume that $\gamma_1,\gamma_2$ exist and both $T_1$ and $T_2$ are nonempty.
  Now we may write any Weil generator $\alpha_i$ in $K_i$ as in Equation~\eqref{eq:alpha_u_gamma_eta_a}, that is, there is a unique triple $(u_i,\eta_i,a_i) \in \sO_{F_i}^\times \times T_i \times \ZZ$ such that
  \[\alpha_i = \frac{u_i(\gamma_i - \overline{\gamma}_i) + \eta_i + a_i}{2}. \]
  For each pair $(\eta_1,\eta_2) \in T_1 \times T_2$, let $X_{\eta_1,\eta_2}$ denote the affine variety cut out by the following equations in $\QQ[x_1,x_2,y]$:
  \begin{align}
    \prod_{\sigma: F_1 \to \CC} \left( (\sigma(\eta_1) + x_1)^2 - 4y \right) &= \frac{\Disc_{K_1}}{\Disc_{F_1}^2}
    \label{eq:111}
    \\
    \prod_{\tau: F_2 \to \CC} \left( (\tau(\eta_2) + x_2)^2 - 4y \right) &= \frac{\Disc_{K_2}}{\Disc_{F_2}^2}
    \label{eq:222}
    \\
    \label{eq:333}
    \prod_{\genfrac{}{}{0pt}{}{\sigma: F_1 \to \CC}{\tau: F_2 \to \CC}} \left((\sigma(\eta_1) + x_1) - (\tau(\eta_2) + x_2)\right)^2 &= 1.
  \end{align}
  By \cite{gyory1976} the sets $T_i$ are finite, so there are a finite number of the $X_{\eta_1,\eta_2}$.
  Let~$\mathcal{X}$ denote their union.

  We will now construct a finite-to-one map $\psi$ from Weil generators of $K_1 \times K_2$ to $\mathcal{X}(\ZZ)$.
  Let $(\alpha_1,\alpha_2)$ be a Weil generator in $K_1 \times K_2$.
  Then $\alpha_i$ corresponds to a unique triple $(u_i,\eta_i,a_i) \in \sO_{F_i}^\times \times T_i \times \ZZ$.
  Let $q = \alpha_1\overline{\alpha}_1 = \alpha_2\overline{\alpha}_2$.
  Define $\psi$ by sending $(\alpha_1,\alpha_2) \mapsto (a_1,a_2,q)$.
  The image of $\psi$ satisfies Equations~\eqref{eq:111} and \eqref{eq:222} by Lemma~\ref{lem:weil-gen-to-curve}, and Equation~\eqref{eq:333} by Corollary~\ref{cor:products-of-wg}.
  Next we will show that $\psi$ is finite-to-one.
  It suffices to show that for a given $a_i$ and $q$ in $\ZZ$, there are finitely many $u_i \in \sO_{F_i}^\times$ and $\eta_i \in T_i$ such that $\alpha_i = (u_i(\gamma_i - \overline{\gamma}_i) + \eta_i + a_i)/2$ is a Weil generator for $K_i$ with $\alpha_i\overline{\alpha}_i = q$.
  Every such Weil generator satisfies $4\alpha_i\overline{\alpha}_i = 4q = u_i^2\Norm_{K_i/F_i}(\gamma_i - \overline{\gamma}_i) + (\eta_i + a_i)^2$.
  In particular, $u_i$ is determined up to sign by $a_i$, $q$, and $\eta_i$.
  Since $T_i$ is finite, there are only finitely many possible $\eta_i$.
  Hence there are only finitely many possible $u_i$ as well.

  It suffices to show that each component $X_{\eta_1,\eta_2}$ has dimension $0$, as this implies that $\mathcal{X}(\ZZ)$ is finite.
  Let $(\eta_1,\eta_2) \in T_1 \times T_2$, and let $X = X_{\eta_1,\eta_2}$.
  Notice that Equation~\eqref{eq:333} is a polynomial equation in $x_1 - x_2$.
  Therefore there are complex numbers $\beta_k$ for $k=1,\ldots,\tilde k,$ with $\tilde k=2[K_1:\QQ][K_2:\QQ]$, such that
  \begin{equation}
  \label{eq:444}
    \left( \prod_{\genfrac{}{}{0pt}{}{\sigma: F_1 \to \CC}{\tau: F_2 \to \CC}} \left(x_1-x_2 + \sigma(\eta_1) - \tau(\eta_2)\right)^2\right) -1 =
    \prod_{k=1}^{\tilde k} (x_1-x_2-\beta_k).
  \end{equation}

  Hence we can write $X$ as a union of subvarieties $X_k$, for $k=1,\ldots,\tilde k$, where $X_k$ is the intersection of the geometrically irreducible surfaces (see Lemma~\ref{lem:curve-is-irreducible}) defined by Equations~\eqref{eq:111} and \eqref{eq:222}  and by
  \begin{equation*}
       x_1-x_2-\beta_k =0.
  \end{equation*}

  In particular we can eliminate $x_2$ and describe $X_k$ as the intersection of the two irreducible plane curves
  \begin{align}
    \label{eq:11}
    \prod_{\sigma: F_1 \to \CC} \left( (\sigma(\eta_1) + x_1)^2 - 4y \right) - \frac{\Disc_{K_1}}{\Disc_{F_1}^2} &= 0,
    \\
    \label{eq:22}
    \prod_{\tau: F_2 \to \CC} \left( (\tau(\eta_2) + x_1 - \beta_k)^2 - 4y \right) - \frac{\Disc_{K_2}}{\Disc_{F_2}^2} &= 0.
  \end{align}
  By B{\'e}zout's theorem, the intersection of two geometrically irreducible plane curves has dimension $0$ as long as the curves are distinct.
  So it is sufficient to show that the polynomials in Equations~\eqref{eq:11} and \eqref{eq:22} are not proportional.
  We will prove this by contradiction.

  Suppose that the two polynomials in Equations~\eqref{eq:11} and \eqref{eq:22} are proportional.
  By comparing the highest degree monomial in $y$ we must have that Equations~\eqref{eq:11} and \eqref{eq:22} are equal. Note that this also implies $[K_1:\QQ] = [K_2:\QQ]$.
  Setting $x_1 = -(\sigma(\eta_1) + \tau(\eta_2) - \beta_k)/2$ and $y = (\sigma(\eta_1) + x_1)^2/4$ for any choice of $\sigma$ and $\tau$ will zero out the products in Equations~\eqref{eq:11} and $\eqref{eq:22}$, which shows that $\Disc_{K_1}/\Disc_{F_1}^2 = \Disc_{K_2}/\Disc_{F_2}^2$.
  Setting $y=0$ leads to
  \[
    \prod_{\sigma: F_1 \to \CC} \left( \sigma(\eta_1) + x_1 \right)^2
    =
    \prod_{\tau: F_2 \to \CC} \left( \tau(\eta_2) + x_1 - \beta_k \right)^2.
  \]
  By comparing the zeros we see that, for each $k$, there is a bijection $\sigma \mapsto \tau_{\sigma,k}$ such that $\sigma(\eta_1) = \tau_{\sigma,k}(\eta_2) - \beta_k$.
  Fix a $1\leq k_0 \leq \tilde k$ substituting $\sigma(\eta_1) = \tau_{\sigma,k_0}(\eta_2) - \beta_{k_0}$ and  $x_2=x_1-\beta_{k_0}$ in Equation~\eqref{eq:444} yields
  \begin{align*}
    \prod_{k=1}^{\tilde k} \left(x_1-(x_1-\beta_{k_0})-\beta_k\right)
    &=
    \left( \prod_{\sigma,\tau} x_1-(x_1-\beta_{k_0}) + \tau_{\sigma,k_0}(\eta_2) - \beta_{k_0} - \tau(\eta_2)\right)^2 - 1
    \\
    0
    &=
    \left( \prod_{\sigma,\tau} \left(\tau_{\sigma,k}(\eta_2) - \tau(\eta_2)\right)^2\right)- 1
  \end{align*}
  Since the product is taken over all embeddings $\sigma$ (and $\tau$) we have that the right hand side is equal to $-1$, which is a contradiction.
\end{proof}

\begin{theorem}\label{thm:finite-weil-gens-in-product}
  Let $K = K_1 \times \cdots \times K_n$ be a product of CM fields. If $n > 1$ then $K$ has finitely many Weil generators.
\end{theorem}
\begin{proof}
  This follows directly from Lemma~\ref{lem:wg-products} and the fact that every Weil generator for $K$ projects to a Weil generator for each product $K_i \times K_j$, with $i\neq j$.
\end{proof}

\section{Products of super-isolated varieties}
\label{sec:siav-products}

In this section, we study products of super-isolated abelian varieties.
Recall that the \emph{Weil polynomial} $h$ of a $d$-dimensional abelian variety $A/\FF_q$ is defined as the characteristic polynomial of the Frobenius endomorphism of $A$.
The polynomial $h$ lies in $\ZZ[x]$, it has degree $2d$ and all its complex roots have absolute value $\sqrt{q}$.
Also we define the \emph{real Weil polynomial} of $A$ the unique polynomial $g$ in $\ZZ[x]$ such that $h(x)=x^{d}g(x+q/x)$. In particular if $h$ is irreducible and $\pi$ is a root of $h$, then $g$ is the minimal polynomial of $\pi+q/\pi$.

\subsection{Glueing exponent and super isolated abelian varieties}
\label{sec:glueing_exp}

\begin{lemma}
\label{lemma:super-iso_from_prod}
Let $A_1$ and $A_2$ be abelian varieties over a finite field $\FF_q$. Assume that $A_1$ and $A_2$ have no isogeny factor in common. If the product $A_1 \times A_2$ is super-isolated, then both $A_1$ and $A_2$ are super-isolated.
\end{lemma}
\begin{proof}
Let $B_1$ and $B_2$ be abelian varieties isogenous to $A_1$ and $A_2$, respectively.
Since $A_1\times A_2$ is super-isolated then there exists an isomorphism $\varphi:A_1\times A_2 \to B_1\times B_2$. Observe that
\[ \Hom_{\FF_q}(A_1 \times A_2 , B_1 \times B_2) = \Hom_{\FF_q}(A_1 , B_1) \times \Hom_{\FF_q}( A_2 , B_2), \]
because $A_1$ and $A_2$ have no isogeny factor in common. Hence there are isomorphisms $\varphi_1:A_1\to B_1$ and $\varphi_2:A_2\to B_2$ such that $\varphi = \varphi_1\times \varphi_2$. In particular $A_1$ and $A_2$ are super-isolated.
\end{proof}

\begin{definition}[{cf.~\cite[Def.~2.1]{HoweLauter12}}]
Let $A_1$ and $A_2$ be abelian varieties over $\FF_q$. The \emph{glueing exponent} $e(A_1,A_2)$ of $A_1$ and $A_2$ is the greatest common divisor of the exponent of $\Delta$, where $\Delta$ ranges over all finite group schemes that embed in both a variety isogenous to $A_1$ and a variety isogenous to $A_2$.
\end{definition}

\begin{lemma}
\label{lemma:super-iso_to_prod}
Let $A_1$ and $A_2$ be abelian varieties over a finite field $\FF_q$.
If $A_1$ and $A_2$ are super-isolated and $e(A_1,A_2)=1$ then the product $A_1 \times A_2$ is super-isolated.
\end{lemma}
\begin{proof}
It is a direct application of \cite[Lemma~2.3]{HoweLauter12}.
\end{proof}

Observe that $A_1$ and $A_2$ have no isogeny factor in common if and only if $e(A_1,A_2)<\infty$. See also \cite[after Def.~2.1]{HoweLauter12}.
In view of this consideration, Lemmas \ref{lemma:super-iso_from_prod} and \ref{lemma:super-iso_to_prod} allow us to understand super-isolated products in terms of the glueing exponent.
In general the glueing exponent is tricky to compute.
Nevertherless one can use \cite[Prop.2.8]{HoweLauter12} to show that
for abelian varieties $A$ and $B$ over $\FF_q$ with no isogeny factor in common
the glueing exponent divides the resultant of the radicals of the real Weil polynomials $g_A$ and $g_B$ of $A$ and $B$.
In particular if such a resultant is $1$ then also $e(A,B)=1$.
In particular Corollary \ref{cor:products-of-wg} can be considered as a reformulation in terms of Weil generators (restricted to \goodenough{} abelian varieties, see Definition~\ref{def:good-enough} below) of Lemma \ref{lemma:super-iso_to_prod}.

In what follows we will use a different strategy:
We will restrict ourselves to a subcategory of abelian varieties and give conditions for them (and their products) to be super-isolated in terms of Weil generators.
The upshot of this approach is that it is more computationally friendly, since Weil generators can be enumerated.
Moreover, in such a subcategory, we will be able to study powers of super-isolated abelian varieties.

\subsection{Weil generators and super isolated abelian varieties}
\label{sec:weil_gens}
From now on we will focus on squarefree varieties.
That is, abelian varieties whose Weil polynomial is squarefree.
Our main tool is the equivalence of categories between certain abelian varieties $A/\FF_q$ and $\ZZ$-modules with extra structure.

\begin{definition}
\label{def:good-enough}
  Let $A/\FF_q$ be an abelian variety, and let $h$ denote its Weil polynomial.
  Let $p$ be the characteristic of $\FF_q$.
  Then $A$ is \emph{\goodenough{}} if the following holds:
  \begin{enumerate}
    \item the polynomial $h$ factors into distinct irreducible factors.
    \item the polynomial $h$ has no real roots.
    \item the polynomial $h$ is ordinary (meaning half of the roots of $h$ are $p$-adic units) or $q$ is prime.
  \end{enumerate}
\end{definition}
Suppose that $A/\FF_q$ is \goodenough{}.
Then there is an isogeny
\[ A \sim A_1\times\ldots\times A_n, \]
where the $A_i$ are simple and pairwise non-isogenous abelian varieties.
By Honda-Tate theory, see \cite{Honda68} and \cite{Tate66}, the Weil polynomial factors  into $h = h_1 \cdots h_n$, where $h_i$ is the Weil polynomial of $A_i$.
Because of the ordinary condition (or the condition that $q$ is prime), the polynomials $h_i$ are irreducible.
In particular, if we let $\pi_i$ denote a root of $h_i$ and put $K_i = \QQ(\pi_i)$, then $K_i$ is a CM field of degree $2\dim A_i$.
Let $K = \prod K_i$ and consider the subring $\ZZ[\pi,\overline{\pi}] \subseteq K$ where $\pi = (\pi_1,\dots,\pi_n)$ and $\overline{\pi} = (\overline{\pi}_1,\dots,\overline{\pi}_n)$ with $\overline{\pi}_i = q/\pi_i$.
We will identify $K$ with $\End A \otimes \QQ$.
Let $\sO_K = \prod \sO_{K_i}$, that is, the integral closure of $\ZZ$ in $K$.
Note that $\sO_K$ and $\ZZ[\pi,\overline{\pi}]$ are free finitely generated $\ZZ$-modules of rank $\dim_\QQ(K)$.
In particular we have a finite-index inclusion $\ZZ[\pi,\overline{\pi}] \subseteq \sO_K$ and
one can show that $\ZZ[\pi,\overline{\pi}] \subseteq \End A \subseteq \sO_K$.

By \cite[Thm.~4.3]{marseglia2021squarefree} the \emph{\goodenough{}} abelian varieties can be functorially described in terms of fractional $\ZZ[\pi,\overline{\pi}]$-ideals.
Such a description builds on the equivalences of categories established in \cite{deligne1969varieties} and \cite{centeleghe2015categories}.

\begin{theorem}\label{thm:siav}
  If $A/\FF_q$ is \goodenough{} and simple, then $A$ is super-isolated if and only if $K$ has class number $1$ and $\ZZ[\pi,\overline{\pi}] = \sO_K$.
  Equivalently, $A$ is super-isolated if and only if $K$ has class number $1$ and $\pi$ is a Weil generator for $K$.
\end{theorem}
\begin{proof}
  The result for ordinary varieties is given in \cite[Thm.~7.4]{scholl2019super}.
  The extension to the case where $q$ is prime is \cite[Thm.~16]{scholl2018super}.
\end{proof}

\begin{theorem}\label{thm:product-siav}
  If $A/\FF_q$ is \goodenough{}, then $A$ is super-isolated if and only if each $A_i$ is super-isolated and $\ZZ[\pi,\overline{\pi}] = \sO_K$.
  Equivalently, $A$ is super-isolated if and only if $\Pic(\sO_K)$ is trivial and $\pi$ is a Weil generator for $K$.
\end{theorem}
\begin{proof}
  By \cite[Thm.~4.3]{marseglia2021squarefree} we have that $A$ is super-isolated if and only if
  $\ZZ[\pi,\overline{\pi}] = \sO_K$ and $\Pic(\sO_K)$ is trivial.
  Moreover by Theorem \ref{thm:siav} we have that each simple factor $A_i$ is  super-isolated if and only if $\ZZ[\pi_i,\overline{\pi}_i] = \sO_{K_i}$ and $\Pic(\sO_{K_i})$ is trivial.
  The result follows from the equality $\Pic(\sO_K)=\prod\Pic(\sO_{K_i})$.
\end{proof}

By using Theorem \ref{thm:product-siav} it is easy to construct super-isolated products of big dimension as shown in the following example.
\begin{example}
  Consider the polynomials
  \begin{align*}
    h_1(x) &= (x^4 - 2x^3 + 3x^2 - 4x + 4), \\
    h_2(x) &= (x^6 - 4x^5 + 9x^4 - 15x^3 + 18x^2 - 16x + 8), \\
    h_3(x) &= (x^6 - 3x^5 + 6x^4 - 9x^3 + 12x^2 - 12x + 8), \\
    h_4(x) &= (x^8 - 5x^7 + 12x^6 - 20x^5 + 29x^4 - 40x^3 + 48x^2 - 40x + 16), \\
    h_5(x) &= (x^8 - 5x^7 + 13x^6 - 25x^5 + 39x^4 - 50x^3 + 52x^2 - 40x + 16), \\
    h_6(x) &= (x^8 - 4x^7 + 5x^6 + 2x^5 - 11x^4 + 4x^3 + 20x^2 - 32x + 16).
  \end{align*}
  One can check that each factor $h_i$ determines a Weil generator $\alpha_i$ for a CM field $K_i=\QQ[x]/h_i$, and that $(\alpha_1,\ldots,\alpha_6)$ is a Weil generator for the product $K_1 \times \cdots \times K_6$.
  Therefore $h=\prod_i h_i$ is the Weil polynomial for a $20$-dimensional super-isolated variety over $\FF_2$.
\end{example}

\begin{example}
  Let $A/\FF_5$ be the product of the elliptic curves $E_1: y^2 = x^3 + 4x + 2$ and $E_2: y^2 = x^3 + 3x + 2$.
  The Weil polynomials of $E_1$ and $E_2$ are $h_1 = x^2 - 3x + 5$ and $h_2 = x^2 - x + 5$, respectively.
  Using Corollary~\ref{cor:products-of-wg} and Theorem~\ref{thm:product-siav}, it is straightforward to check that $E_1$ and $E_2$ are super-isolated, but $A$ is not, see also Example~\ref{ex:product-of-weil-generators-is-not-enough}.
  Therefore $A$ must be isogeneous, but not isomorphic, to some other abelian surface over $\FF_5$.
  One way to verify this claim is to observe that
  $\ZZ[\pi,\bar{\pi}] \subsetneq \sO_K=\ZZ[\pi_1]\times \ZZ[\pi_2]=\End_{\FF_5}(A).$
  By \cite[Thm.~4.3]{marseglia2021squarefree} we have that there exists an abelian variety $A'$ isogenous to $A$ with $\End_{\FF_5}(A')=\ZZ[\pi,\bar{\pi}]$.
  In particular $A$ is not isomorphic to $A'$.
  More precisely, since $\sO_K$ is the unique over-order of $\ZZ[\pi,\bar{\pi}]$ and $\Pic(\ZZ[\pi,\bar{\pi}])\simeq \ZZ/3\ZZ$,
  by using \cite[Thm.~4.3]{marseglia2021squarefree} we can conclude that the isogeny class of $A$ contains exactly $4$ isomorphism classes of abelian varieties, $3$ of which have endomorphism ring $\ZZ[\pi,\bar{\pi}]$.
\end{example}

\begin{corollary}\label{cor:finite-products-siav}
  Let $g$ be a positive integer.
  There are only finitely many \goodenough{} super-isolated abelian varieties of dimension $g$ that are not simple.
  In particular, there are only finitely many finite fields $\FF_q$ for which such a variety may exist.
\end{corollary}
\begin{proof}
  By Theorem~\ref{thm:product-siav} it is sufficient to count Weil generators in products of CM fields with class number $1$.
  In \cite[Thm.~2, p.~136]{stark1974effective}, Stark showed that for any fixed degree, there are finitely many CM fields with class number~$1$ with that degree.
  Hence there are finitely many products $K_1 \times \cdots \times K_n$ with $\sum \deg K_i = 2g$ such that each $K_i$ is a CM field with class number $1$.
  By Theorem~\ref{thm:finite-weil-gens-in-product}, any such product of fields contains only finitely many Weil generators.
\end{proof}

\subsection{Products of super-isolated elliptic curves and abelian surfaces}
\label{sec:enum_siav}

In this section we outline a general strategy to enumerate certain products of super-isolated varieties.

\begin{algorithm}[h]
  \caption{Enumerate Weil Generators}\label{alg:find-wg-2}
  \begin{algorithmic}[1]
    \REQUIRE A product $K$ of CM fields $K_1 $ and $ K_2$, with maximal totally real subfields $F_1$ and $F_2$, respectively.
    \ENSURE All Weil generators for $K$.
    \STATE $F \gets F_1 \times F2$
    \STATE Find $\gamma_i$ such that $\sO_K = \sO_{F}[\gamma_i]$
    \STATE $T_i \gets $ a complete set of $\eta_i \in F_i$ such that $\sO_{F_i} = \ZZ[\eta_i]$ up to integer translation
    \FORALL{$(\eta_1,\eta_2) \in T_1 \times T_2$}
      \STATE $X_{\eta_1,\eta_2} \gets$ the variety defined in the proof of Lemma~\ref{lem:wg-products}
      \FORALL{$P \in X_{\eta_1,\eta_2}(\QQ)$}
        \FORALL{$u_i \in \sO_{F_i}^\times$ with $u_i^2\Norm_{K/F}(\gamma_i - \overline{\gamma}_i) + (\eta_i + P_{x_i})^2 = 4P_y$}
        \STATE $\alpha_i \gets \left(u_i(\gamma_i - \overline{\gamma}_i) + \eta_i + P_{x_i}\right)/2$
        \IF{$\alpha_1 \in \sO_{K_1}$ and $\alpha_2 \in \sO_{K_2}$}
          \PRINT $(\alpha_1,\alpha_2)$
        \ENDIF
        \ENDFOR
      \ENDFOR
    \ENDFOR
  \end{algorithmic}
\end{algorithm}

\begin{proposition}
\label{prop:alg1}
  Let $K$ be a product of CM fields $K_1 $ and $ K_2$, with maximal totally real subfields $F_1$ and $F_2$, respectively.
  Algorithm~\ref{alg:find-wg-2} exactly outputs all Weil generators for $K_1 \times K_2$.
\end{proposition}
\begin{proof}
  As seen in the proof of Lemma~\ref{lem:wg-products}, every Weil generator $(\alpha_1,\alpha_2)$ of $K$ corresponds to a rational point on one of the varieties $X_{\eta_1,\eta_2}$ for some $\eta_1,\eta_2 \in T_1 \times T_2$.
  The algorithm ranges over all rational points on all such varieties and finds all possible pairs $(\alpha_1,\alpha_2)$ which could be Weil generators.
  This means that the output will include all Weil generators.

  It remains to show that everything the algorithm outputs is a Weil generator for $K_1 \times K_2$.
  Suppose that the algorithm outputs $(\alpha_1,\alpha_2)$ from a point on $X_{\eta_1,\eta_2}$ for some $(\eta_1,\eta_2) \in T_1 \times T_2$.
  It is straightforward to check that if $\alpha_i \in \sO_{K_i}$ then it is indeed a Weil generator for $K_i$ using \cite[Lem.~3.13]{scholl2019super}.
  By construction of $X_{\eta_1,\eta_2}$, it follows that $(\alpha_1,\alpha_2)$ satisfies the resultant condition in Corollary~\ref{cor:products-of-wg}.
  Therefore $(\alpha_1,\alpha_2)$ is a Weil generator for $K_1 \times K_2$.
\end{proof}

\begin{remark}
\label{rmk:alg1}
  Algorithm~\ref{alg:find-wg-2} can be extended to arbitrary products $K_1 \times \cdots \times K_n$ in a straightforward way:
  First compute the set $W_{i,j}$ of all Weil generators in the sub-product $K_i \times K_j$ for all $1 \leq i,j \leq n$.
  Then search for tuples $(\alpha_1,\dots,\alpha_n)$ with $(\alpha_i,\alpha_j) \in W_{i,j}$ for all $1 \leq i,j \leq n$.
\end{remark}

Using Tables of CM fields of degree $2$ and $4$ with class number one, we can enumerate all \goodenough{} superisolated abelian varieties that factor into products of elliptic curves and surfaces.
\begin{corollary}
  There are $240$ \goodenough{} super-isolated abelian varieties that factor into a product of curves and surfaces.
  A summary of some of their characteristics are given in Table~\ref{tab:siav-data}.
  \footnote{The raw data collected and source code can be found at \url{https://github.com/tscholl2/siav-polarizations-products}.}
\end{corollary}
\begin{proof}
  We implemented Algorithm~\ref{alg:find-wg-2} in {\tt\nocite{sage}Sage} and found all \goodenough{} super-isolated abelian varieties that decompose into a product of curves and surfaces.
  This was done by first finding a complete list of CM fields with class number $1$ and degree $\leq 4$ (see \cite[Tbl.~3]{scholl2019super} for references).
  For each pair $K_i,K_j$ of such fields, we computed the set of Weil generators $W_{i,j}$ in $K_i \times K_j$.
  We filtered the Weil generators whose minimal polynomials satisfy the conditions in Definition~\ref{def:good-enough}, as these correspond to simple \goodenough{} varieties.
  Next we organized the data into a graph $\mathcal{G}$ as follows.
  The vertex set of $\mathcal{G}$ is given by all Weil generators appearing as part of a pair in some $W_{i,j}$.
  We add an edge between $\alpha_1$ and $\alpha_2$ if $(\alpha_1,\alpha_2) \in W_{i,j}$ for some $i,j$.
  That is, the edge set of $\mathcal{G}$ is the union of the $W_{i,j}$.
  Finally, we used standard methods to enumerate all complete subgraphs (cliques) of size $\geq 2$ in $\mathcal{G}$.
\end{proof}

\begin{table}[!ht]
  \begin{center}
  \begin{tabular}{|c||c|c|c|c|c|}\hline
  $q$ & $1 \times 1$ & $1 \times 2$ & $1 \times 1 \times 2$ & $1 \times 2 \times 2$ & $2 \times 2$ \\\hline\hline
  $2$ & $4$ & $24$ & $10$ & $12$ & $18$ \\\hline
  $3$ & $4$ & $24$ & $6$ & $12$ & $18$ \\\hline
  $4$ & $ $ & $2$ & $ $ & $ $ & $ $ \\\hline
  $5$ & $2$ & $12$ & $ $ & $2$ & $6$ \\\hline
  $7$ & $ $ & $8$ & $ $ & $ $ & $ $ \\\hline
  $8$ & $ $ & $2$ & $ $ & $ $ & $ $ \\\hline
  $9$ & $ $ & $2$ & $ $ & $ $ & $ $ \\\hline
  $11$ & $2$ & $8$ & $2$ & $4$ & $4$ \\\hline
  $13$ & $ $ & $6$ & $ $ & $ $ & $ $ \\\hline
  $17$ & $2$ & $8$ & $2$ & $ $ & $ $ \\\hline
  $19$ & $ $ & $ $ & $ $ & $ $ & $2$ \\\hline
  $32$ & $ $ & $2$ & $ $ & $ $ & $ $ \\\hline
  $41$ & $ $ & $2$ & $ $ & $ $ & $ $ \\\hline
  $47$ & $ $ & $4$ & $ $ & $ $ & $ $ \\\hline
  $59$ & $ $ & $2$ & $ $ & $ $ & $ $ \\\hline
  $61$ & $ $ & $2$ & $ $ & $ $ & $ $ \\\hline
  $83$ & $2$ & $ $ & $ $ & $ $ & $ $ \\\hline
  $101$ & $2$ & $ $ & $ $ & $ $ & $ $ \\\hline
  $173$ & $ $ & $2$ & $ $ & $ $ & $ $ \\\hline
  $227$ & $2$ & $ $ & $ $ & $ $ & $ $ \\\hline
  $257$ & $2$ & $ $ & $ $ & $ $ & $ $ \\\hline
  $283$ & $ $ & $2$ & $ $ & $ $ & $ $ \\\hline
  $383$ & $ $ & $2$ & $ $ & $ $ & $ $ \\\hline
  $1523$ & $2$ & $ $ & $ $ & $ $ & $ $ \\\hline
  $1601$ & $2$ & $ $ & $ $ & $ $ & $ $ \\\hline
  $18131$ & $ $ & $2$ & $ $ & $ $ & $ $ \\\hline
  \end{tabular}
  \end{center}
  \caption{The number of \goodenough{} super-isolated products of elliptic curves and abelian surfaces over each finite field $\FF_q$. Each column represents a decomposition type, so for example $1 \times 1$ represents the product of two non-isogenous elliptic curves. If the cell is empty, it means there is no super-isolated abelian variety over $\FF_q$ with the prescribed decomposition type.}
  \label{tab:siav-data}
\end{table}

\begin{remark}
  Observe that by \cite[Cor.~7.6]{scholl2019super}, for every fixed dimension $g>2$ we have only finitely many (ordinary) super-isolated abelian varieties.
  This was shown by mapping Weil generators to integral points on degree $g$ plane curves.
  Given a way to enumerate the integral points on such a curve, it is straightforward to enumerate super-isolated abelian varieties of dimension $g$.
  However, computing integral points on high degree curves (which often have high genus) is a difficult problem.
  If instead we restrict to enumerating non-trivial products of super-isolated $g$-folds, then Algorithm~\ref{alg:find-wg-2} only requires enumerating points on a dimension $0$ variety (of degree $\approx g^3$).
  This seems to suggest that finding products of super-isolated varieties is easier than finding singletons.
\end{remark}

\subsection{Powers of super-isolated abelian varieties}

In the case where $A$ is isogenous to a power, we can apply the results from \cite{marseglia2019power} to prove the following Theorem.

\begin{theorem}\label{thm:powers}
  Let $A/\FF_q$ be an \goodenough{} abelian variety.
  If $A$ is super-isolated, then $A^n$ is super-isolated for every $n\geq 1$.
  Conversely, if there exists $n\geq 1$ such that $A^n$ is super-isolated then $A$ is super-isolated.
\end{theorem}
\begin{proof}
  By Theorem \ref{thm:siav} we have that $A$ is super-isolated if and only if $K$ has class number $1$ and $\ZZ[\pi,\overline{\pi}] = \sO_K$.
  By Steinitz theory, see \cite{Steinitz11}, this is equivalent to having a unique isomorphism class of torsion-free $\ZZ[\pi,\overline{\pi}]$-modules of rank $n$ (for any $n\geq 1$), which is represented by
  \[
    \sO_K^n=\sO_K\oplus\ldots\oplus\sO_K.
  \]
  Using the classification given in \cite[Theorem 4.1]{marseglia2019power}, one sees that  this happens if and only if $A^n$ is super-isolated.
\end{proof}

\section{Principal polarizations}
\label{sec:siav-pol}

In general, (principal) polarizations of abelian varieties in an ordinary square-free isogeny class can be computed up to polarized isomorphisms using \cite[Alg.~3]{marseglia2021squarefree}.
However, in this section we show that if the isogeny class is super-isolated, then the situation is much simpler.
For a CM type $\Phi$ of a CM field $K$ we denote by $N_\Phi$ the associated norm, that is,
\[ N_\Phi(\alpha) = \prod_{\varphi \in \Phi} \varphi(\alpha), \]
for every $\alpha \in K$.

\subsection{Existence of principal polarizations}
\label{sec:pol-existence}
\begin{lemma}\label{lem:type-norm-diff-1}
  Let $K$ be a CM field of degree $2g$ with maximal totally real subfield $F$.
  Let $\Phi$ be a CM type of $K$, and $\alpha \in K$ satisfying $\sO_K = \sO_{F}[\alpha]$.
  Then the following statements are equivalent:
  \begin{enumerate}
      \item \label{lem:type-norm-diff-1:1} $K/F$ is unramified at all finite primes.
      \item \label{lem:type-norm-diff-1:2} $\alpha - \overline{\alpha} \in \sO_K^\times$.
      \item \label{lem:type-norm-diff-1:3} $\Norm_{K/\QQ}(\alpha - \overline{\alpha}) = 1$.
      \item \label{lem:type-norm-diff-1:4} $N_\Phi(\alpha - \overline{\alpha}) = \pm 1$.
  \end{enumerate}
  Moreover, if any of the statements holds then $g$ is even.
\end{lemma}
\begin{proof}
  Observe that we have
  \begin{equation}
  \label{eq:norms}
    \Norm_{K/\QQ}(\alpha - \overline{\alpha})
    =
    N_\Phi(\alpha - \overline{\alpha})N_\Phi(\overline{\alpha} - \alpha)
    =
    (-1)^gN_\Phi(\alpha - \overline{\alpha})^2.
  \end{equation}
  By \cite[Ch.~III, Prop.~2.4]{neukrich1999algebraic}, $(\alpha - \overline{\alpha})\sO_K$ is the relative different ideal $\Diff_{K/F}$.
  So $K/F$ is unramified at all finite primes if and only if $\alpha - \overline{\alpha} \in \sO_K^\times$.
  As all norms in CM fields are non-negative, this is equivalent to $\Norm_{K/\QQ}(\alpha - \overline{\alpha}) = 1$.
  Hence we have proved that \eqref{lem:type-norm-diff-1:1}, \eqref{lem:type-norm-diff-1:2} and \eqref{lem:type-norm-diff-1:3} are equivalent.
  Using Equation \eqref{eq:norms} we get that \eqref{lem:type-norm-diff-1:3} is also equivalent to
  \[
    1
    = (-1)^gN_\Phi(\alpha - \overline{\alpha})^2.
  \]
  If $K/F$ is unramified at all finite primes, then $g$ is even by \cite[Lem.~10.2]{howe1995principally} and we get that $N_\Phi(\alpha - \overline{\alpha}) = \pm 1$.
  Conversely if $N_\Phi(\alpha - \overline{\alpha}) = \pm 1$ then by Equation \eqref{eq:norms} we get $\Norm_{K/\QQ}(\alpha - \overline{\alpha}) = (-1)^g$.
  Since norms are positive in CM-fields, we deduce that $g$ is even and $K/F$ is unramified. This concludes the proof.
\end{proof}

Let $A/\FF_q$ be a simple ordinary super-isolated abelian variety of dimension $g$.
Let $h$ be the Weil polynomial of $A$.
Put $K=\QQ[x]/h=\QQ(\pi)$.
Fix an isomorphism $j:\overline{\QQ}_p \simeq \CC$ and
let $\Phi$ be the set of embeddings $\phi: K \to \CC$ such that $\nu(\phi(\pi)) > 0$, where $\nu$ is the $p$-adic valuation on $\CC$ induced by $j$.
Since $A$ is ordinary, precisely half of the roots of $h$ are $p$-adic units.
Therefore $\Phi$ is a CM-type of $K$.

\begin{lemma}\label{lem:pp-iff-type-norm-m1}
  The abelian variety $A$ does not admit a principal polarization if and only if $N_\Phi(\pi - \overline{\pi}) = -1$.
\end{lemma}
\begin{proof}
  Suppose $A$ does not admit a principal polarization.
  Then by
  \cite[Cor.~11.4]{howe1995principally} we have that
  $K/F$ is unramified at all finite primes and $N_{\Phi}(\pi - \overline{\pi}) < 0$.
  Because $\pi$ is a Weil generator for $K$, it follows that $\sO_K = \sO_{F}[\pi]$ by \cite[Lem.~3.13]{scholl2019super}.
  So by Lemma~\ref{lem:type-norm-diff-1}, we have $N_\Phi(\pi - \overline{\pi}) = -1$.

  Conversely, suppose that $N_\Phi(\pi - \overline{\pi}) = -1$.
  By Lemma ~\ref{lem:type-norm-diff-1} it follows that $K/F$ is unramified at all finite primes, and that $(\pi-\bar \pi) \in \sO_K^\times$.
  In particular there is no prime of $F$ dividing $(\pi-\bar \pi)$ which is inert in $K/F$.
  It now follows from \cite[Cor.~11.4]{howe1995principally} that $A$ does not admit a principal polarization.
\end{proof}
\begin{theorem}\label{thm:siav-pp-criteria}
  The abelian variety $A$ does not admit a principal polarization if and only if $\Norm_{K/\QQ}(\pi - \overline{\pi}) = 1$ and the middle coefficient $a_g$ of the Weil polynomial is $-1 \bmod{q}$ if $q > 2$ and $-1 \bmod{4}$ if $q = 2$.
\end{theorem}
\begin{proof}
  By Lemmas~\ref{lem:type-norm-diff-1} and \ref{lem:pp-iff-type-norm-m1}, it is enough to show that if $K/\QQ$ is unramified at all finite primes, then $N_\Phi(\pi - \overline{\pi}) = -1$ if and only if $a_g$ satisfies the congruence condition in the statement.
  By \cite[Prop.~11.5]{howe1995principally}, if $K/F$ is unramified at all finite primes then $N_\Phi(\pi - \overline{\pi}) \equiv a_g \bmod{q}$ if $q > 2$ and $N_\Phi(\pi - \overline{\pi}) \equiv a_g \bmod{4}$ if $q = 2$.
  As these moduli are enough to distinguish $\pm 1$, the result follows.
\end{proof}

\begin{example}
  Let $\pi$ denote a root of $h(x) = x^8 + x^7 - 3x^6 - x^5 + 7x^4 - 2x^3 - 12x^2 + 8x + 16$.
  Then $\pi$ is a Weil generator for the CM field $K = \QQ(\pi)$ of degree $8$.
  So $\pi$ corresponds to a super-isolated abelian fourfold $A$ over $\FF_2$.
  One can check that $\Norm_{K/\QQ}(\pi - \overline{\pi}) = 1$.
  Also, the middle coefficient of $h(x)$ is $7 \equiv -1 \bmod{4}$, so $A$ does not admit a principal polarization.
\end{example}

\begin{example}
  We will show that the two conditions in Theorem~\ref{thm:siav-pp-criteria} are independent.
  Let $\alpha_1$ denote a root of $h_1 = x^2 - x + 3$ and $\alpha_2$ a root of $h_2 = x^4 - 5x^2 + 9$.
  Then $\alpha_1$ and $\alpha_2$ are Weil generators for the CM fields $K_i = \QQ(\alpha_i)$.
  Let $q_i = \alpha_i\overline{\alpha}_i$.
  One can compute that $\Norm_{K_1/\QQ}(\alpha_1-\overline{\alpha}_1) = 11$ and $\Norm_{K_2/\QQ}(\alpha_2-\overline{\alpha}_2) = 1$.
  The middle coefficients of $h_1$ and $h_2$ are $-1 \bmod{q_1}$ and $1 \bmod{q_2}$ respectively.
  Therefore $\alpha_1$ satisfies the second hypothesis but not the first of Theorem~\ref{thm:siav-pp-criteria}, while $\alpha_2$ satisfies the first but not the second.
\end{example}

\begin{remark}
  One can show that any abelian variety admits a principal polarization if and only if its quadratic twist does.
  But for super-isolated varieties, this follows almost immediately from Theorem~\ref{thm:siav-pp-criteria}.
  If $h(x)$ is the Weil polynomial of $A/\FF_q$, then $h(-x)$ is the Weil polynomial for its twist.
  Then because $\Norm_{K/\QQ}(\pi - \overline{\pi}) = \Norm_{K/\QQ}((-\pi) - (-\overline{\pi}))$ and $h(x)$ and $h(-x)$ share the same middle coefficient, the criterion in Theorem~\ref{thm:siav-pp-criteria} holds for one if and only if it holds for the other.
\end{remark}

Given any super-isolated abelian variety $A$, we can always construct one with a principal polarization using Zarhin's trick and Theorem~\ref{thm:powers}.

\begin{proposition}
  Let $A$ be a super-isolated abelian variety.
  Then $A^8$ is principally polarized.
\end{proposition}
\begin{proof}
  Since $A$ is super-isolated it is isomorphic to its dual $A^\vee$.
  Hence there is an isomorphism $\varphi:A^8 \to (A\times A^\vee)^4$.
  By Zahrin's trick (see \cite{zarhin1974remark}), the product $(A\times A^\vee)^4$ admits a principal polarization $\mu$.
  Then $\varphi^*\mu = \varphi^\vee \circ \mu \circ \varphi$ is a principal polarization of $A^8$, see for example \cite[p.143]{mumford2008}.
\end{proof}

\subsection{Uniqueness of principal polarizations}
\label{sec:pol-unique}

In this section, we prove that if a super-isolated abelian variety admits a principal polarization, then such a polarization is unique up to polarized isomorphism.

First, we set the following notation.
For a number field $L$, we let $U_L$ denote the group of units of $\sO_L$, and the group of totally positive units is $U_L^+$.
The Hilbert class field of $L$ is denoted by $H_L$, and the narrow Hilbert class field is $H_L^+$.
Recall that $H_L$ is the maximal unramified abelian extension of $L$ and that $H_L^+$ is the maximal abelian extension of $L$ unramified outside of the infinite primes\footnote{
A real archimedean prime ramifies in an extension if it extends to a non-real embedding \cite[Ch.~5.C, p.~94]{cox2011primes}.} (see \cite[Ch.~5.C, Thm.~5.18]{cox2011primes} and \cite[Ch.~8, p.~167]{marcus2018number}).
Moreover, the Galois group $\Gal(H_L/L)$ is isomorphic to the class group of $L$, see \cite[Ch.~5.C, Thm.~5.23]{cox2011primes}.

\begin{lemma}\label{lem:unit-index-narrow-index}
  If $F$ is a totally real field, then
  \[
    [U_F^+ : U_F^2] = [H_F^+ : H_F].
  \]
\end{lemma}
\begin{proof}
  If $r$ is the number of real embeddings of $F$, then $2^r = [U_F : U_F^+][H_F^+ : H_F]$ by \cite[Thm.~3.1, p.~242]{janusz1996algebraic}.
  Since $F$ is totally real, $2^r = [U_F : U_F^2]$ as $U_F \cong \ZZ/2\ZZ \times \ZZ^{r-1}$.
\end{proof}

\begin{lemma}\label{lem:norm-surjective-on-+}
  Let $K$ be a CM field with class number $1$, and let $F$ be the maximal totally real subfield of $K$.
  Then $\Norm_{K/F}(U_K) = U_F^+$.
\end{lemma}
\begin{proof}
  By \cite[Ch.~8, Ex.~27, p.~176]{marcus2018number}, $H_F^+K \subseteq H_K^+$.
  Since $K$ has no real embeddings because it is a CM field, we have $H_K^+ = H_K$ because the infinite primes of $K$ are unramified in every extension.
  Moreover, the assumption that $K$ has class number $1$ means that $H_K = K$.
  Therefore we have that $F \subseteq H_F^+ \subseteq K$, and either $H_F^+ = F$ or $H_F^+ = K$.
  By Lemma~\ref{lem:unit-index-narrow-index}, $U_F^2$ has index $1$ or $2$ in $U_F^+$ depending on which equality holds.

  If $K/F$ is ramified at a finite prime, then $H_F^+ = F$.
  Hence by Lemma~\ref{lem:unit-index-narrow-index} we have $U_F^+ = U_F^2$, and the conclusion follows from the fact that 
  \[U_F^2 \subseteq \Norm_{K/F}(U_K) \subseteq U_F^+.\]

  It remains to consider the case where $K/F$ is unramified.
  Here we have that $H_F^+ = K$, so $[U_F^+ : U_F^2] = 2$.
  To prove the claim, we will show that there is a unit $u \in U_K$ such that $\Norm_{K/F}(u) \notin U_F^2$.

  By Kummer theory, $K = F(\sqrt{-\varepsilon})$ for some totally positive element $\varepsilon \in F$.
  We claim that for every prime $\mathfrak{p}$ of $F$, the valuation $v_{\mathfrak{p}}(\epsilon)$ is even.
  To see this, let $\mathfrak{P}$ be a prime of $K$ lying over $\mathfrak{p}$.
  Since $\mathfrak{p}$ is unramified in $K/F$, then we have $v_{\mathfrak{p}}(\epsilon) = v_{\mathfrak{P}}(\epsilon)$.
  The latter is even because $v_{\mathfrak{P}}(\epsilon) = 2v_{\mathfrak{P}}(\sqrt{-\epsilon})$.

  By above, we can write $\varepsilon\sO_F = \mathfrak{p}_1^{2e_1} \cdots \mathfrak{p}_k^{2e_k}$ for primes $\mathfrak{p}_1,\dots,\mathfrak{p}_k$ of $F$.
  Because $K$ has class number $1$, so does $F$, see \cite[Thm.~4.10]{washington1997introduction}.
  So there are generators $\beta_i$ for each $\mathfrak{p}_i$, and we may write $\varepsilon = \varepsilon'\beta_1^{2e_1}\cdots\beta_k^{2e_k}$ for some $\varepsilon' \in U_F$.
  The extension $F(\sqrt{-\varepsilon})/F$ depends only on the residue of $\varepsilon$ in $F^\times/(F^\times)^2$.
  Therefore we may assume that $\varepsilon = \varepsilon'$, i.e. that $\varepsilon \in U_F$.

  Note that $\Norm_{K/F}(\sqrt{-\varepsilon}) = \varepsilon$.
  If $\varepsilon \notin U_F^2$, then we are done.
  Otherwise, $F(\sqrt{-\varepsilon}) = F(\sqrt{-1})$, so we may assume that $\varepsilon = 1$.
  In particular, it remains to consider the case where $K = F(i)$.

  Let $\zeta$ be a primitive $2^n$th root of unity in $K$ such that $n$ is maximal (by above we know that $n \geq 2$).
  Then we have a tower of the form
  \[
  \begin{tikzcd}
    & K \arrow[dl,-] \arrow[dr,-]
    \\
    \QQ(\zeta) \arrow[dr,-] & & F \arrow[dl,-]
    \\
    & \QQ(\zeta + \overline{\zeta}).
  \end{tikzcd}
  \]
  Recall that since $\zeta$ is a primitive $2^n$th root of unity, the rational prime $2$ is totally ramified in $\QQ(\zeta)$.
  Moreover, $2\sO_{\QQ(\zeta)} = \mathfrak{b}^{2^{n-1}}$ where $\mathfrak{b}$ is the prime ideal generated by $b=1-\zeta$ \cite[Ch.~2, Ex.~34, p.~34]{marcus2018number}.
  In particular $\mathfrak{b}=\overline{\mathfrak{b}}$.
  Put $\mathfrak{a}=\mathfrak{b}\cap \QQ(\zeta + \overline{\zeta})$, which is the unique prime ideal of $\QQ(\zeta + \overline{\zeta})$ above $2$.
  Let $a = b\overline{b}$.
  Observe that $a \in \QQ(\zeta + \overline{\zeta})$ and that $a$ is a generator of $\mathfrak{a}\sO_{\QQ(\zeta)} = \mathfrak{b}^2$.
  Therefore $v_{\mathfrak{b}}(a) = 2v_{\mathfrak{a}}(a)$, and $v_{\mathfrak{b}}(a) = 2$ by construction.
  Hence $a$ is a generator of $\mathfrak{a}$.
  As mentioned above, the ramification index $e(\mathfrak{b}/\mathfrak{a})=2$.
  So, since $K/F$ is unramified, we deduce that there exists an ideal $\mathfrak{c}$ of $F$ such that $\mathfrak{c}^2=\mathfrak{a}\sO_F$.
  Moreover since $F$ has class number $1$, the ideal $\mathfrak{c}$ admits a generator $c\in F$.
  Then $c^2 = au$ for some $u \in U_F$.
  Moreover we have that
  \[ c^2\sO_K=\mathfrak{c}^2\sO_K=\mathfrak{a}\sO_K=\mathfrak{b}^2\sO_K=b^2\sO_K. \]
  Hence $c/b \in U_K$. Also, by construction,
  \[
    \Norm_{K/F}\left(\frac{c}{b}\right) = \frac{au}{a} = u.
  \]
  We conclude the proof by proving that $u \notin U_F^2$.
  Indeed, assume that $u=u_0^2$ for some $u_0\in U_F$. Then
  \[
    \left(\frac{b u_0}{c}\right)^2 = \frac{b^2}{a} = \frac{b\overline{b}(-\zeta)}{a} = -\zeta.
  \]
   This implies that $bu_0/c$ is a primitive $2^{n+1}$ root of unity in $K$, contradicting the maximality of $n$.
\end{proof}
In the next example we show that the assumption on the class number of $K$ is necessary for Lemma \ref{lem:norm-surjective-on-+} to hold.

\begin{example}\label{ex:cl_num_1_necessary}
  Let $K = \QQ[x]/(x^4 - x^3 + x^2 - 3x + 9)$, which has class number $2$.
  Then $K$ is a CM field with maximal totally real subfield $F = \QQ(\sqrt{21})$.
  A fundamental unit of $F$ is $\epsilon = (5 - \sqrt{21})/2$, which is totally positive.
  One can show that $U_K = U_F$.
  Therefore $\epsilon \notin \Norm_{K/F}(U_K) = \langle \epsilon^2 \rangle$, so $U_F^+/\Norm_{K/F}(U_K) \cong \ZZ/2\ZZ$.
\end{example}

\begin{theorem}\label{thm:unique-pol}
  Let $A$ be a simple super-isolated ordinary abelian variety over $\FF_q$ which admits a principal polarization.
  Then the polarization is unique up to polarized isomorphism.
\end{theorem}
\begin{proof}
  We have that $\End(A)=\sO_K$ for a CM-number field $K$ with class number~$1$.
  By \cite[Thm.~5.4.(a)]{marseglia2021squarefree},
  two principal polarizations of $A$ differ by a totally positive unit in $U_{F}$.
  By \cite[Thm.~5.4.(b)]{marseglia2021squarefree},
  two principal polarizations of $A$ are isomorphic if and only if they differ by an element of the form $v\overline{v}$ for a unit $v\in U_K$.
  In other words, the number of principal polarizations of $A$ up to isomorphism is given by the size of the quotient
  \[
    \frac{ U_{F}^+ }{ \Norm_{K/F}(U_K) },
  \]
  and this is trivial by Lemma~\ref{lem:norm-surjective-on-+}.
\end{proof}

\subsection{Products of principal polarizations}

Let $K$ be a CM field.
Recall that for a CM type $\Phi$ of $K$ we say that a totally imaginary element $\lambda \in K$ is $\Phi$-positive if $\Im(\varphi(\lambda))>0$ for every $\varphi \in \Phi$.
Also, for a fractional ideal $I$ of some order $R$ in $K$, we denote by $I^t$ its trace dual ideal, which is defined as
\[
  I^t = \lbrace z \in K : \Trace_{K/\QQ}(xI) \subseteq \ZZ \rbrace.
\]
Also, we define $\overline{I}$ as the image of $I$ by the CM involution of $K$.
\begin{corollary}
\label{cor:pol_square_free}
  Let $A$ be an ordinary squarefree super-isolated abelian variety, say $A=\prod_{1}^n A_i$ with $A_i$ simple.
  Then $A$ admits a principal polarization if and only if each $A_i$ does.
  If this is the case, the principal polarization is unique up to polarized isomorphism.
\end{corollary}
\begin{proof}
  Let $h$ (resp.~$h_i$) be the Weil polynomial of $A$ (resp.~$A_i$).
  Put $K=\QQ[x]/h$ and $K_i=\QQ[x]/h_i$ for each $i$, so that $K = \prod_{i=1}^n K_i$.
  For each $i$ let $F_i$ be the maximal totally real subfield of $K_i$ and put $F=\prod_i F_i$.
  Since the abelian variety $A$ is ordinary, by \cite[Thm.~5.4]{marseglia2021squarefree} we have that $A$ is principally polarized if and only there exist a  $\lambda \in K^*$ which is totally imaginary, $\Phi$-positive, and such that
  \begin{equation}
  \label{eq:pols}
  \lambda\sO_K = \overline{\sO}_K^t.
  \end{equation}
  Since $\sO_K = \oplus_i \sO_{K_i}$ and $\overline{\sO}_K^t = \oplus_i \overline{\sO}_{K_i}^t$,
  if we write $\lambda=(\lambda_1,\ldots,\lambda_n)$ with $\lambda_i$ in $K_i$, then $\lambda$ is totally imaginary and $\Phi$-positive if and only if the same holds for each $\lambda_i$. Also, Equation \eqref{eq:pols} holds if and only if
  \[ \lambda_i\sO_{K_i} = \overline{\sO}_{K_i}^t \]
  for each $i$.
  The statement about the uniqueness follows from the equality
  \[ \frac{ U_{F}^+ }{ \Norm_{K/F}(U_K) } = \prod_{i=1}^n \frac{ U_{F_i}^+ }{ \Norm_{K_i/F_i}(U_{K_i}) }. \]
\end{proof}

\begin{remark}\label{remark:pol_power}
  Let $A$ be an ordinary simple super-isolated abelian variety admitting a principal polarization.
  Then also $A^n$ admits a principal polarization, but this is in general not unique, see \cite[Ex.~6.5, Ex.~6.6]{marseglia2019power}.
\end{remark}

\section{Jacobians}
\label{sec:jacobians}

\begin{proposition}
\label{prop:jac}
Let $C$ and $C'$ be smooth, projective and geometrically integral curves of genus $g>1$ defined over $\FF_q$ with the same zeta function.
Assume that $\Jac(C)$ is ordinary, \goodenough{}, and super-isolated.
Then the curves $C$ and $C'$ are isomorphic.
\end{proposition}
\begin{proof}
Observe that by assumption $\Jac(C')$ is isogenous to $\Jac(C)$, and hence isomorphic since $\Jac(C)$ is super-isolated.
Denote by $\theta$ and $\theta'$ the canonical principal polarizations of $\Jac(C)$ and $\Jac(C')$, respectively.
By Theorem \ref{thm:unique-pol} we deduce that $(\Jac(C),\theta)$ is isomorphic to $(\Jac(C'),\theta')$.
Therefore by Torelli's Theorem we deduce that $C\simeq C'$.
\end{proof}

The next example shows that given a Weil generator $\pi$ for a number field $K=\QQ(\pi)$ with non-trivial class group, we can have two non-isomorphic Jacobians as polarized abelian varieties which are isomorphic as unpolarized abelian varieties in the isogeny class determined by the minimal polynomial of $\pi$.
For more examples and a general method to construct such curves see \cite{Howe96constr_jac}.

\begin{example}
\label{ex:non_siav_jac}
Consider the hyperelliptic curves over $\FF_3$ defined by
\[ C_1: y^2 = 2x^5 + 2x^4 + x^3 + 2x^2 + 1 \text{ and } C_2: y^2 = 2x^5 + x^4 + x + 1. \]
Observe (or use {\tt Magma}\nocite{magma} to verify) that $C_1$ and $C_2$ are not isomorphic.
Their Jacobians lie in the same isogeny class, which is determined by the Weil polynomial
\[ h=x^4 - x^3 + x^2 - 3x + 9. \]
Let $K=\QQ[x]/(h)=\QQ(\pi)$, which is the same field as in Example \ref{ex:cl_num_1_necessary}.
Note that $\pi$ is a Weil generator for $K$ but that the isogeny class is not super-isolated because the class group of $K$ has order two.
Using \cite[Thm.~4.3]{marseglia2021squarefree} we deduce that there are two isomorphism classes of abelian varieties in the isogeny class, represented by say $A$ and $B$.
Using \cite[Thm.~5.4]{marseglia2021squarefree}, we compute that one of the isomorphism classes admits two non-isomorphic principal polarizations, say $(A,\theta_1)$ and $(A,\theta_2)$, while $B$ is not principally polarized.
Note that it is not surprising that $A$ has two non-isomorphic polarizations: indeed
by \cite[Thm.~5.4]{marseglia2021squarefree} the number of non-isomorphic polarizations equals the size of $U_F^+/\Norm_{K/F}(U_K)$ and in Example \ref{ex:cl_num_1_necessary} we showed that it is $2$.
Denote by $\theta'_1$ (resp.~$\theta'_2$) the canonical polarization of $\Jac(C_1)$ (resp.~$\Jac(C_2)$). We deduce that, after possibly relabelling $\theta_1$ and $\theta_2$, we have isomorphisms $(\Jac(C_1),\theta'_1)\simeq (A,\theta_1)$ and $(\Jac(C_2),\theta'_2)\simeq (A,\theta_2)$.
In particular we have $\Jac(C_1)\simeq \Jac(C_2) \simeq A$ as unpolarized abelian varieties.
\end{example}

\end{document}